\documentclass{amsart}
\usepackage{graphicx, amsmath, amssymb, hyperref}
\newtheorem{thm}{Theorem}[section]

\newtheorem{cor}[thm]{Corollary}

\newtheorem{prop}[thm]{Proposition}
\newtheorem*{thm1}{Theorem}
\newtheorem{claim}{Claim}[thm]

\theoremstyle{definition}
\newtheorem{defn}[thm]{Definition}
\newtheorem{ex}[thm]{Example}
\theoremstyle{remark}

\newtheorem{rem}[thm]{Remark}
\newtheorem{q}[thm]{Question}
\numberwithin{equation}{section}
\newcommand{\norm}[1]{\left\Vert#1\right\Vert}
\newcommand{\abs}[1]{\left\vert#1\right\vert}
\newcommand{\set}[1]{\left\{#1\right\}}
\newcommand{\defined}[1]{\textit{#1}}
\newcommand{\mc}[1]{\mathcal{#1}}
\newcommand{\op}[1]{\operatorname{#1}}
\newcommand{\IL}{\mc{L}_{\omega_1, \omega}}
\newcommand{\Aut}{\op{Aut}}
\newcommand*{\sneg}{\mathord{\sim}}
\begin{document}

\title{Expressive power of infinitary $[0, 1]$-valued logics}%
\author{Christopher J.~Eagle}%
\address{Department of Mathematics and Statistics, University of Victoria, PO BOX 1700 STN CSC, Victoria, B.C., Canada, V8W 2Y2}%
\email{eaglec@uvic.ca}%

\date{\today}%
\begin{abstract}
We consider model-theoretic properties related to the expressive power of three analogues of $\IL$ for metric structures. We give an example showing that one of these infinitary logics is strictly more expressive than the other two, but also show that all three have the same elementary equivalence relation for complete separable metric structures.  We then prove that a continuous function on a complete separable metric structure is automorphism invariant if and only if it is definable in the more expressive logic.  Several of our results are related to the existence of Scott sentences for complete separable metric structures.
\end{abstract}
\maketitle
\section{Introduction}
In the last several years there has been considerable interest in the \emph{continuous first-order logic for metric structures} introduced by Ben Yaacov and Usvyatsov in the mid-2000s and published in \cite{BenYaacov2010}.  This logic is suitable for studying structures based on metric spaces, including a wide variety of structures encountered in analysis.  Continuous first-order logic is a generalization of first-order logic and shares many of its desirable model-theoretic properties, including the compactness theorem.  While earlier logics for considering metric structures, such as Henson's logic of positive bounded formulas (see \cite{Henson2003}), were equivalent to continuous first-order logic, the latter has emerged as the current standard first-order logic for developing the model theory of metric structures.  The reader interested in a detailed history of the interactions between model theory and analysis can consult \cite{Iovino2014}.

In classical discrete logic there are many examples of logics that extend first-order logic, yet are still tame enough to allow a useful model theory to be developed; many of the articles in \cite{Barwise1985} describe such logics.  The most fruitful extension of first-order logic is the infinitary logic $\IL$, which extends the formula creation rules from first-order to also allow countable conjunctions and disjunctions of formulas, subject only to the restriction that the total number of free variables remains finite.  While the compactness theorem fails for $\IL$, it is nevertheless true that many results from first-order model theory can be translated in some form to $\IL$ - see \cite{Keisler1971} for a thorough development of the model theory of $\IL$ for discrete structures.

Many properties considered in analysis have an infinitary character.  It is therefore natural to look for a logic that extends continuous first-order logic by allowing infinitary operations.  In order to be useful, such a logic should still have desirable model-theoretic properties analogous to those of the discrete infinitary logic $\IL$.  There have recently been proposals for such a logic by Ben Yaacov and Iovino \cite{BenYaacov2009}, Sequeira \cite{Sequeira}, and the author \cite{Eagle2014}; we call these logics $\IL^C$, $\IL^C(\rho)$, and $\IL$, respectively.  The superscript $C$ is intended to emphasize the continuity of the first two of these logics, in a sense to be described below.  The goal of Section \ref{sec:Infinitary} is to give an overview of some of the model-theoretic properties of each of these logics, particularly with respect to their expressive powers.  Both $\IL$ and $\IL^C$ extend continuous first-order logic by allowing as formulas some expressions of the form $\sup_n \phi_n$, where the $\phi_n$'s are formulas.  The main difference between $\IL^C$ and $\IL$ is that the former requires infinitary formulas to define \emph{uniformly continuous} functions on all structures, while the latter does not impose any continuity requirements.  Allowing discontinuous formulas provides a significant increase in expressive power, including the ability to express classical negation (Proposition \ref{prop:ExactNegation}), at the cost of a theory which is far less well-behaved with respect to metric completions (Example \ref{Comparison:ex:EEExample}).  The logic $\IL^C(\rho)$ is obtained by adding an additional operator $\rho$ to $\IL^C$, where $\rho(x, \phi)$ is interpreted as the distance from $x$ to the zeroset of $\phi$.  We show in Theorem \ref{thm:SameLogic} that $\rho$ can be defined in $\IL$.

One of the most notable features of the discrete logic $\IL$ (in a countable signature) is that for each countable structure $M$ there is a sentence $\sigma$ of $\IL$ such that a countable structure $N$ satisfies $\sigma$ if and only if $N$ is isomorphic to $M$.  Such sentences are known as \emph{Scott sentences}, having first appeared in a paper of Scott \cite{Scott1965}.  In Section \ref{sec:Scott} we discuss some consequences of the existence of Scott sentences for complete separable metric structures.  The existence of Scott sentences for complete separable metric structures was proved by Sequeira \cite{Sequeira} in $\IL^C(\rho)$ and Ben Yaacov, Nies, and Tsankov \cite{BenYaacov2014} in $\IL^C$.  Despite having shown in Section \ref{sec:Infinitary} that the three logics we are considering have different expressive powers, we use Scott sentences to prove the following in Proposition \ref{prop:SameEERelation}:

\begin{thm1}
Let $M$ and $N$ be separable complete metric structures in the same countable signature.  The following are equivalent:
\begin{itemize}
\item{
$M \cong N$,
}
\item{
$M \equiv N$ in $\IL^C$,
}
\item{
$M \equiv N$ in $\IL^C(\rho)$,
}
\item{
$M \equiv N$ in $\IL$.
}
\end{itemize}
\end{thm1}

Scott's first use of his isomorphism theorem was to prove a definability result, namely that a predicate on a countable discrete structure is automorphism invariant if and only if it is definable by an $\IL$ formula.  The main new result of this note is a metric version of Scott's definability theorem (Theorem \ref{Infinitary:Scott:thmScottDefinability}):

\begin{thm1}
Let $M$ be a separable complete metric structure, and $P : M^n \to [0, 1]$ be a continuous function.  The following are equivalent:
\begin{itemize}
\item{
$P$ is invariant under all automorphisms of $M$,
}
\item{
there is an $\IL$ formula $\phi(\vec{x})$ such that for all $\vec{a} \in M^n$, $P(\vec{a}) = \phi^M(\vec{a})$.
}
\end{itemize}
\end{thm1}

The proof of the above theorem relies heavily on replacing the constant symbols in an $\IL$ sentence by variables to form an $\IL$ formula; Example \ref{ex:Substitute} shows that this technique cannot be used in $\IL^C$ or $\IL^C(\rho)$, so our method does not produce a version of Scott's definability theorem in $\IL^C$ or $\IL^C(\rho)$.

\subsection*{Acknowledgements}
Some of the material in this paper appears in the author's Ph.\,D.~thesis \cite{Eagle2015a}, written at the University of Toronto under the supervision of Jos\'e Iovino and Frank Tall.  We thank both supervisors for their suggestions and insights, both on the work specifically represented here, and on infinitary logic for metric structures more generally.  We also benefited from discussions with Ilijas Farah and Bradd Hart, which led to Theorem \ref{thm:SameLogic} and Example \ref{ex:CStar}.  The final version of this paper was completed during the Focused Research Group ``Topological Methods in Model Theory" at the Banff International Research Station.  We thank BIRS for providing an excellent atmosphere for research and collaboration, and we also thank Xavier Caicedo, Eduardo Du\'e\~{n}ez, Jos\'e Iovino, and Frank Tall for their comments during the Focused Research Group.

\section{Infinitary logics for metric structures}\label{sec:Infinitary}
Our goal is to study infinitary extensions of first-order continuous logic for metric structures.  To begin, we briefly recall the definition of metric structures and the syntax of first-order continuous logic.  The reader interested in an extensive treatment of continuous logic can consult the survey \cite{BenYaacov2008a}.  
\begin{defn}
A \defined{metric structure} is a metric space $(M, d^M)$ of diameter at most $1$, together with:
\begin{itemize}
\item{
A set $(f_i^M)_{i \in I}$ of uniformly continuous \emph{functions} $f_i : M^{n_i} \to M$,
}
\item{
A set $(P_j^M)_{j \in J}$ of uniformly continuous \emph{predicates} $P_j : M^{m_j} \to [0, 1]$,
}
\item{
A set $(c_k^M)_{k \in K}$ of distinguished \emph{elements} of $M$.
}
\end{itemize}
We place no restrictions on the sets $I, J, K$, and frequently abuse notation by using the same symbol for a metric structure and its underlying metric space.
\end{defn}

Metric structures are the semantic objects we will be studying.  On the syntactic side, we have \emph{metric signatures}.  By a \defined{modulus of continuity} for a uniformly continuous function $f : M^n \to M$ we mean a function $\delta : \mathbb{Q} \cap (0, 1) \to \mathbb{Q} \cap (0, 1)$ such that for all $a_1, \ldots, a_n, b_1, \ldots, b_n \in M$ and all $\epsilon \in \mathbb{Q} \cap (0, 1)$,
\[\sup_{1 \leq i \leq n}d(a_i, b_i) < \delta(\epsilon) \implies  d(f(a_i), f(b_i)) \leq \epsilon.\]
Similarly, $\delta$ is a modulus of continuity for $P : M^n \to [0, 1]$ means that for all $a_1, \ldots, a_n, b_1, \ldots, b_n \in M$,
\[\sup_{1 \leq i \leq n}d(a_i, b_i) < \delta(\epsilon) \implies  \abs{P(a_i) - P(b_i)} \leq \epsilon.\]

\begin{defn}
A \defined{metric signature} consists of the following information:
\begin{itemize}
\item{
A set $(f_i)_{i \in I}$ of \emph{function symbols}, each with an associated arity and modulus of uniform continuity,
}
\item{
A set $(P_j)_{j \in J}$ of \emph{predicate symbols}, each with an associated arity and modulus of uniform continuity,
}
\item{
A set $(c_k)_{k \in K}$ of \emph{constant symbols}.
}
\end{itemize}
When no ambiguity can arise, we say ``signature" instead of ``metric signature".
\end{defn}

When $S$ is a metric signature and $M$ is a metric structure, we say that $M$ is an \defined{$S$-structure} if the distinguished functions, predicates, and constants of $M$ match the requirements imposed by $S$.  Given a signature $S$, the \defined{terms} of $S$ are defined recursively, exactly as in the discrete case.

\begin{defn}\label{def:Formula}
Let $S$ be a metric signature.  The \defined{$S$-formulas} of continuous first-order logic are defined recursively as follows.
\begin{enumerate}
\item{
If $t_1$ and $t_2$ are terms, then $d(t_1, t_2)$ is a formula.
}
\item{
If $t_1, \ldots, t_n$ are $S$-terms, and $P$ is an $n$-ary predicate symbol, then $P(t_1, \ldots, t_n)$ is a formula.
}
\item{
If $\phi_1, \ldots, \phi_n$ are formulas, and $f : [0, 1]^n \to [0, 1]$ is continuous, then $f(\phi_1, \ldots, \phi_n)$ is a formula.  We think of each such $f$ as a \emph{connective}.
}
\item{
If $\phi$ is a formula and $x$ is a variable, then $\inf_x \phi$ and $\sup_x \phi$ are formulas.  We think of $\sup_x$ and $\inf_x$ as \emph{quantifiers}.
}
\end{enumerate}
\end{defn}

Given a metric structure $M$, a formula $\phi(\vec{x})$ of the appropriate signature, and a tuple $\vec{a} \in M$, we define the value of $\phi$ in $M$ at $\vec{a}$, denoted $\phi^M(\vec{a})$, in the obvious recursive manner.  We write $M \models \phi(\vec{a})$ to mean $\phi^M(\vec{a}) = 0$.  The basic notions of model theory are then defined in the expected way by analogy to discrete first-order logic.

The only difference between our definitions and those of \cite{BenYaacov2008a} is that in \cite{BenYaacov2008a} it is assumed that the underlying metric space of each structure is complete.  We do not want to make the restriction to complete metric spaces in general, so our definition of structures allows arbitrary metric spaces, and we speak of \emph{complete metric structures} when we want to insist on completeness of the underlying metric.  In first-order continuous logic there is little lost by considering only complete metric structures, since every structure is an elementary substructure of its metric completion.  This is also true in $\IL^C$ and $\IL^C(\rho)$, but not in $\IL$, as Example \ref{Comparison:ex:EEExample} below illustrates.

In continuous logic the connectives $\max$ and $\min$ play the roles of $\wedge$ and $\vee$, respectively, in the sense that for a metric structure $M$, formulas $\phi(\vec{x})$ and $\psi(\vec{x})$, and a tuple $\vec{a}$, we have $M \models \max\{\phi(\vec{a}), \psi(\vec{a})\}$ if and only if $M \models \phi(\vec{a})$ and $M \models \psi(\vec{a})$, and similarly for $\min$ and disjunction.  Consequently, the most direct adaptation of $\IL$ to metric structures is to allow the formation of formulas $\sup_n \phi_n$ and $\inf_n \phi_n$, at least provided that the total number of free variables remains finite (the restriction on the number of free variables is usually assumed even in the discrete case).  However, one of the important features of continuous logic is that it is a \emph{continuous} logic, in the sense that each formula $\phi(x_1, \ldots, x_n)$ defines a continuous function $\phi^M : M^n \to [0, 1]$ on each structure $M$.  The pointwise supremum or infimum of a sequence of continuous functions is not generally continuous.

A second issue arises from the fact that one expects the metric version of $\IL$ to have the same relationship to separable metric structures as $\IL$ has to countable discrete structures.  Separable metric structures are generally not countable, so some care is needed in arguments whose discrete version involves taking a conjunction indexed by elements of a fixed structure.  For instance, one standard proof of Scott's isomorphism theorem is of this kind (see \cite[Theorem 1]{Keisler1971}).  Closely related to the question of whether or not indexing over a dense subset is sufficient is the issue of whether the zeroset of a formula is definable.

With the above issues in mind, we present some of the infinitary logics for metric structures that have appeared in the literature.  The first and third of the following logics were both called ``$\IL$" in the papers where they were introduced, and the second was called ``$\IL(\rho)$"; we add a superscript ``$C$" to the first and second logics to emphasize that they are continuous logics.

\begin{defn}\label{def:InfFormula}
The three infinitary logics for metric structures we will be considering are:
\begin{itemize}
\item{
$\IL^C$ (Ben Yaacov-Iovino \cite{BenYaacov2009}): Allow formulas $\sup_{n < \omega}\phi_n$ and $\inf_{n < \omega}\phi_n$ as long as the total number of free variables remains finite, and the formulas $\phi_n$ satisfy a common modulus of uniform continuity.
}
\item{
$\IL^C(\rho)$ (Sequeira \cite{Sequeira}): Extend $\IL^C$ by adding an operator $\rho(x, \phi)$, interpreted as the distance from $x$ to the zeroset of $\phi$.
}
\item{
$\IL$ (Eagle \cite{Eagle2014}): Allow formulas $\sup_{n < \omega}\phi_n$ and $\inf_{n < \omega}\phi_n$ as long as the total number of free variables remains finite, without regard to continuity.
}
\end{itemize}
\end{defn}

The logic $\IL$ was further developed by Grinstead \cite{Grinstead2014}, who in particular provided an axiomatization and proof system.

Other infinitary logics for metric structures which are not extensions of continuous first-order logic have also been studied.  In a sequence of papers beginning with his thesis \cite{Ortiz1997}, Ortiz develops a logic based on Henson's positive bounded formulas and allows infinitary formulas, but also infinite strings of quantifiers.  An early version of \cite{Coskey2014} had infinitary formulas in a logic where the quantifiers $\sup$ and $\inf$ were replaced by category quantifiers.  

\begin{rem}\label{rem:Formulas}
We will often write formulas in any of the above logics in forms intended to make their meaning more transparent, but sometimes this can make it less obvious that the expressions we use are indeed valid formulas.  For example, in the proof of Theorem \ref{thm:SameLogic} below, we will be given an $\IL$ formula $\phi(\vec{x})$, and we will define
\[\rho_{\phi}(\vec{x}) = \inf_{\vec{y}}\min\left\{\left(d(\vec{x}, \vec{y}) + \sup_{n \in \mathbb{N}} \min\{n\phi(\vec{y}), 1\}\right), 1\right\}.\]
The preceding definition can be seen to be a valid formula of $\IL$ as follows.  For each $n \in \mathbb{N}$ define $u_n : [0, 1] \to [0, 1]$ by $u(z) = \min\{nz, 1\}$.  Define $v : [0, 1]^2 \to [0, 1]$ by $v(z, w) = \min\{z+w, 1\}$.  Then each $u_n$ is continuous, as is $v$, and we have
\[\rho_{\phi}(\vec{x}) = \inf_{\vec{y}}v\left(d(\vec{x}, \vec{y}),\,\,\sup_{n\in\mathbb{N}}u_n(\phi(\vec{x}))\right).\]
A similar process may be used throughout the remainder of the paper to see that expressions we claim are formulas can indeed be expressed in the form of Definitions \ref{def:Formula} and \ref{def:InfFormula}.
\end{rem}

The remainder of this section explores some of the relationships between $\IL^C$, $\IL^C(\rho)$, and $\IL$.  It is clear that each $\IL^C$ formula is both an $\IL$ formula and an $\IL^C(\rho)$ formula.  The next result shows that the $\rho$ operation is implemented by a formula of $\IL$, so each $\IL^C(\rho)$ formula is also equivalent to an $\IL$ formula.

\begin{thm}\label{thm:SameLogic}
For every $\IL$ formula $\phi(\vec{x})$ there is an $\IL$ formula $\rho_{\phi}(\vec{x})$ such that for every metric structure $M$ and every $\vec{a} \in M$,
\[\rho_{\phi}^M(\vec{a}) = \inf\{d(\vec{a}, \vec{b}) : \phi^M(\vec{b}) = 0\}.\]
\end{thm}
\begin{proof}
Let $\phi$ be an $\IL$ formula, and define
\[\rho_{\phi}(\vec{x}) = \inf_{\vec{y}}\min\left\{\left(d(\vec{x}, \vec{y}) + \sup_{n \in \mathbb{N}} \min\{n\phi(\vec{y}), 1\}\right), 1\right\}.\]
(See Remark \ref{rem:Formulas} above for how to express this as an official $\IL$ formula).  Now consider any metric structure $M$, and any $\vec{y} \in M$.  We have
\[\sup_{n \in \mathbb{N}}\min\{n\phi^M(\vec{y}), 1\} = \begin{cases} 0 & \text{ if } \phi^M(\vec{y}) = 0, \\ 1 &\text{ otherwise.}\end{cases}\]
Therefore for any $\vec{a}, \vec{y} \in M$,
\[\min\left\{\left(d(\vec{a}, \vec{y}) + \sup_{n \in \mathbb{N}} \min\{n\phi(\vec{y}), 1\}\right), 1\right\} = \begin{cases}d(\vec{a}, \vec{y}) &\text{ if $\phi^M(\vec{y}) = 0$,} \\ 1 &\text{ otherwise.}\end{cases}\]
Since all values are in $[0, 1]$, it follows that:
\[\rho_{\phi}^M(\vec{a}) = \inf\left(\{d(\vec{a}, \vec{y}) : \phi^M(\vec{y}) = 0\} \cup \{1\}\right) = \inf\{d(\vec{a}, \vec{y}) : \phi^M(\vec{y}) = 0\}.\]
\end{proof}

Each formula of $\IL^C$ or $\IL^C(\rho)$ defines a uniformly continuous function on each structure, and just as in first-order continuous logic, the modulus of continuity of this function depends only on the signature, not the particular structure.  By contrast, the functions defined by $\IL$ formulas need not be continuous at all.  The loss of continuity causes complications for the theory, especially when one is interested in \emph{complete} metric structures, as is often the case in applications.  Of particular note is the fact that, while every metric structure is an $\IL^C$-elementary substructure of its metric completion, this is very far from being true for the logic $\IL$:

\begin{ex}\label{Comparison:ex:EEExample}
Let $S$ be the signature consisting of countably many constant symbols $(q_n)_{n < \omega}$.  Consider the $\IL$ formula
\[\phi(x) = \inf_{n < \omega}\sup_{R \in \mathbb{N}}\min\{1, Rd(x, q_n)\}.\]
For any $a$ in a metric structure $M$ we have $M \models \phi(a)$ if and only if $a = q_n$ for some $n$.  In particular, if $M$ is a countable metric space which is not complete, and $(q_n)_{n < \omega}$ is interpreted as an enumeration of $M$, then
\[M \models \sup_x \phi(x) \qquad \text{and} \qquad \overline{M} \not\models \sup_x\phi(x).\]
In particular, $M \not\equiv_{\IL} \overline{M}$.
\end{ex}

While discontinuous formulas introduce complications, they also give a significant increase in expressive power.  As an example, recall that continuous first-order logic lacks an exact negation connective, in the sense that there is no connective $\neg$ such that $M \models \neg\phi$ if and only if $M \not\models \phi$.  Indeed there is no continuous function $\neg : [0, 1] \to [0, 1]$ such that $\neg(x) = 0$ if and only if $x \neq 0$, so $\IL^C$ also lacks an exact negation connective.  Similarly, the formula $\inf_n\phi_n$ is not the exact disjunction of the formulas $\phi_n$, and $\inf_x$ is not an exact existential quantifier, and neither exact disjunction nor exact existential quantification is present in either continuous infinitary logic.  In $\IL$, we recover all three of these classical operations on formulas.

\begin{prop}\label{prop:ExactNegation}
The logic $\IL$ has an exact countable disjunction, an exact negation, and an exact existential quantifier.
\end{prop}
\begin{proof}
We first show that $\IL$ has an exact infinitary disjunction.  Suppose that $(\phi_n(\vec{x}))_{n < \omega}$ are formulas of $\IL$.  Define
\[\psi(\vec{x}) = \inf_{n < \omega}\sup_{R \in \mathbb{N}}\min\{1, R\phi_n(\vec{x})\}.\]
Then in any metric structure $M$, for any tuple $\vec{a}$, we have
\[M \models \psi(\vec{a}) \qquad \iff \qquad M \models \phi_n(\vec{a}) \text{ for some $n$}.\]

Using the exact disjunction we define the exact negation. Given any formula $\phi(\vec{x})$, define
\[\neg\phi(\vec{x}) = \bigvee_{n < \omega}\left(\phi(\vec{x}) \geq \frac{1}{n}\right),\]
where $\bigvee$ is the exact disjunction described above.  Then for any metric structure $M$, and any $\vec{a} \in M$,
\begin{align*}
M \models \neg\phi(\vec{a}) &\iff (\exists n < \omega) M \models \phi(\vec{a}) \geq \frac{1}{n} \\
&\iff (\exists n < \omega) \phi^M(\vec{a}) \geq \frac{1}{n} \\
&\iff \phi^M(\vec{a}) \neq 0 \\
&\iff M \not\models \phi(\vec{a})
\end{align*}

Finally, with exact negation and the fact that $M \models \sup_{\vec{x}} \phi(\vec{x})$ if and only if $M \models \phi(\vec{a})$ for every $\vec{a} \in M$, we define $\exists \vec{x} \phi$ to be $\neg \sup_{\vec{x}} \neg \phi$, and have that $M \models \exists \vec{x} \phi(\vec{x})$ if and only if there is $\vec{a} \in M$ such that $M \models \phi(\vec{a})$.
\end{proof}

\begin{rem}
Some caution is necessary when using the negation operation defined in Proposition \ref{prop:ExactNegation}.  Consider the following properties a negation connective $\sneg$ could have for all metric structures $M$, all tuples $\vec{a} \in M$, and all formulas $\phi(\vec{x})$.  These properties mimic properties of negation in classical discrete logic:
\begin{enumerate}
\item{
$M \models \sneg\phi(\vec{a})$ if and only if $M \not\models \phi(\vec{a})$,
}
\item{
$M \models \sneg\sneg\phi(\vec{a})$ if and only if $M \models \phi(\vec{a})$,
}
\item{
$(\sneg\sneg\phi)^M(\vec{a}) = \phi^M(\vec{a})$.
}
\end{enumerate}
Properties (1) and (3) each implies property (2).  In classical $\{0, 1\}$-valued logics there is no distinction between properties (2) and (3), but these properties do not coincide for $[0, 1]$-valued logic.  Property (2) is strictly weaker than property (1), since the identity connective $\sneg \sigma = \sigma$ satisfies (2) but not (1).

The connective $\neg$ defined in the proof of Proposition \ref{prop:ExactNegation} has properties (1) and (2), but does not have property (3), because if $\phi(\vec{a})^M > 0$, then $(\neg\neg\phi)^M(\vec{a}) = 1$.  The approximate negation commonly used in continuous first-order logic, which is defined by $\sneg~\phi(\vec{x})~= 1-\phi(\vec{x})$, satisfies properties (2) and (3), but not property (1).

In fact, there is no truth-functional connective in \emph{any} $[0, 1]$-valued logic that satisfies both (1) and (3).  Suppose that $\sneg$ were such a connective.  Then $\sneg~:~[0, 1]~\to~[0, 1]$ would have the following two properties for all $x \in [0, 1]$, as consequences of (1) and (3), respectively:
\begin{itemize}
\item{$\sneg(x) = 0$ if and only if $x \neq 0$,}
\item{$\sneg(\sneg(x)) = x$.}
\end{itemize}
The first condition implies that $\sneg$ is not injective, and hence cannot satisfy the second condition.
\end{rem}

The expressive power of $\IL$ is sufficient to introduce a wide variety of connectives beyond those of continuous first-order logic and the specific ones described in Proposition \ref{prop:ExactNegation}.

\begin{prop}\label{prop:Borel}
Let $u : [0, 1]^n \to [0, 1]$ be a Borel function, with $n < \omega$, and let $(\phi_l(\vec{x}))_{l < n}$ be $\IL$-formulas.  There is an $\IL$-formula $\psi(\vec{x})$ such that for any metric structure $M$ and any $\vec{a} \in M$,
\[\psi^M(\vec{a}) = u(\phi_1^M(\vec{a}), \ldots).\]
\end{prop}
\begin{proof}
Recall that the \emph{Baire hierarchy} of functions $f : [0, 1]^n \to [0, 1]$ is defined recursively, with $f$ being \emph{Baire class $0$} if it is continuous, and \emph{Baire class $\alpha$} (for an ordinal $\alpha > 0$) if it is the pointwise limit of a sequence of functions each from some Baire class $< \alpha$.  The classical Lebesgue-Hausdorff theorem (see \cite[Proposition 3.1.32 and Theorem 3.1.36]{Srivastava1998}) implies that a function $f : [0, 1]^\omega \to [0, 1]$ is Borel if and only if it is Baire class $\alpha$ for some $\alpha < \omega_1$.  Our proof will therefore be by induction on the Baire class $\alpha$ of our connective $u : [0, 1]^n \to [0, 1]$.  The base case is $\alpha = 0$, in which case $u$ is continuous, and hence is a connective of first-order continuous logic.

Now suppose that $u = \lim_{k\to\infty}u_k$ pointwise, with each $u_k$ of a Baire class $\alpha_k < \alpha$.  By induction, for each $k$ let $\psi_k(\vec{x})$ be such that for every metric structure $M$ and every $\vec{a} \in M$ $\psi_k^M(\vec{a}) = u_k(\phi_1^M(\vec{a}), \ldots, \phi_n^M(\vec{a}))$.  Then we have
\begin{align*}
u(\phi_1^M(\vec{a}), \ldots, \phi_n^M(\vec{a})) &= \lim_{k\to\infty}u_k(\phi_1^M(\vec{a}), \ldots, \phi_n^M(\vec{a})) \\
&= \limsup_{k\to\infty}\psi_k^M(\vec{a}) \\
&= \inf_{k\geq 0}\sup_{m\geq k}\psi_m^M(\vec{a})
\end{align*}
The final expression shows that the required $\IL$ formula is $\inf_{k \geq 0}\sup_{m \geq k}\psi_m(\vec{x})$.
\end{proof}

\begin{rem}
The case of Proposition \ref{prop:Borel} for sentences appears, with a different proof, in \cite[Theorem 1.25]{Grinstead2014}.

The expressive power of continuous first-order logic is essentially unchanged if continuous functions of the form $u : [0, 1]^\omega \to [0, 1]$ are allowed as connectives in addition to the continuous functions on finite powers of $[0, 1]$ (see \cite[Proposition 9.3]{BenYaacov2008a}).  If such infinitary continuous connectives are permitted in $\IL$, then the same proof as above also shows that $\IL$ implements all Borel functions $u : [0, 1]^\omega \to [0, 1]$.
\end{rem}

In order to obtain the benefits of both $\IL$ and $\IL^C$ or $\IL^C(\rho)$, it is sometimes helpful to work in $\IL$ and then specialize to a more restricted logic when continuity becomes relevant.  A \defined{fragment} of an infinitary metric logic $\mc{L}$ is a set of $\mc{L}$-formulas including the formulas of continuous first-order logic, closed under the connectives and quantifiers of continuous first-order logic, closed under subformulas, and closed under substituting terms for variables.  In \cite{Eagle2014} we defined a fragment $L$ of $\IL$ to be \defined{continuous} if it has the property that every $L$-formula defines a continuous function on all structures.  The definition of a continuous fragment ensures that if $L$ is a continuous fragment and $M$ is a metric structure, then $M \preceq_L \overline{M}$.

It follows immediately from the definitions that $\IL^C$ is a continuous fragment of both $\IL$ and $\IL^C(\rho)$.  The construction of the $\rho$ operation as a formula of $\IL$ in Theorem \ref{thm:SameLogic} uses discontinuous formulas as subformulas, so $\IL^C(\rho)$ is not a continuous fragment of $\IL$, although it would be if we viewed the formula $\rho_\phi$ from Theorem \ref{thm:SameLogic} as having only $\phi$ as a subformula.  While it is a priori possible that there are continuous fragments of $\IL$ that are not subfragments of $\IL^C(\rho)$, we are not aware of any examples.  It also remains unclear whether or not the $\rho$ operation of $\IL^C(\rho)$ can be implemented by an $\IL^C$ formula.  We therefore ask:

\begin{q}\label{q:SameLogic}
Suppose that $\phi(\vec{x})$ is an $\IL$ formula such that for every subformula $\psi$ of $\phi$, $\psi^M : M^n \to [0, 1]$ is \emph{uniformly} continuous, with the modulus of uniform continuity not depending on $M$.  Is $\phi$ equivalent to an $\IL^C$ formula?  Is $\rho(\vec{y}, \phi)$ equivalent to an $\IL^C$ formula?
\end{q}

A positive answer to the first part of Question \ref{q:SameLogic} would imply that every continuous fragment of $\IL$ is a fragment of $\IL^C$.  In the first part of the question, the answer is negative if we only ask for $\phi$ to define a uniformly continuous function.  For example, consider the sentence $\sigma = \sup_x\phi(x)$ from Example \ref{Comparison:ex:EEExample}.  For any $M$ we have $\sigma^M : M^0 \to [0, 1]$ is constant, yet we saw that this $\sigma$ can be a witness to $M \not\equiv_{\IL} \overline{M}$, and hence is not equivalent to any $\IL^C$ sentence.  This example can be easily modified to produce examples of $\IL$ formulas with free variables that are uniformly continuous but not equivalent to any $\IL^C$ sentence (for example, $\max\{\sigma, d(y, y)\}$).

\section{Consequences of Scott's Isomorphism Theorem}\label{sec:Scott}
The existence of Scott sentences for complete separable metric structures was first proved by Sequeira \cite{Sequeira} in $\IL^C(\rho)$.  Sequeira's proof of the existence of Scott sentences is a back-and-forth argument, generalizing the standard proof in the discrete setting.  An alternative proof of the existence of Scott sentences in $\IL^C$ goes by first proving a metric version of the L\'opez-Escobar Theorem, which characterizes the isomorphism-invariant bounded Borel functions on a space of codes for structures as exactly those functions of the form $M \mapsto \sigma^M$ for an $\IL^C$-sentence $\sigma$.  Using this method Scott sentences in $\IL^C$ were found by Coskey and Lupini \cite{Coskey2014} for structures whose underlying metric space is the Urysohn sphere, and such that all of the distinguished functions and predicates share a common modulus of uniform continuity.  Shortly thereafter, Ben Yaacov, Nies, and Tsankov obtained the same result for all complete metric structures. 

\begin{thm}[{\cite[Corollary 2.2]{BenYaacov2014}}]\label{thm:ScottIsomorphism}
For each separable complete metric structure $M$ in a countable signature there is an $\IL^C$ sentence $\sigma$ such that for every other separable complete metric structure $N$ of the same signature,
\[\sigma^N = \begin{cases}0 &\text{ if $M \cong N$} \\ 1 &\text{ otherwise}\end{cases}\]
\end{thm}

We note that a positive answer to Question \ref{q:SameLogic} would imply that Sequeira's proof works in $\IL^C$, and hence would give a more standard back-and-forth proof of Theorem \ref{thm:ScottIsomorphism}.  

\begin{rem}
Even with the increased expressive power of $\IL$ over $\IL^C$, we cannot hope to prove the existence of Scott sentences for arbitrary (i.e., possibly incomplete) separable metric structures, because there are $2^{2^{\aleph_0}}$ pairwise non-isometric separable metric spaces (\cite[Theorem 2.1]{KellyNordhaus1951}), but only $2^{\aleph_0}$ sentences of $\IL$ in the empty signature.  
\end{rem}

We can easily reformulate Theorem \ref{thm:ScottIsomorphism} to apply to incomplete structures, but little is gained, as we only get uniqueness at the level of the metric completion.

\begin{cor}
For each separable metric structure $M$ in a countable signature there is an $\IL^C$ sentence $\sigma$ such that for every other separable metric structure $N$ of the same signature, 
\[\sigma^N = \begin{cases}0 &\text{ if $\overline{M} \cong \overline{N}$} \\ 1 &\text{ otherwise}\end{cases}\]
\end{cor}
\begin{proof}
Let $\sigma$ be the Scott sentence for $\overline{M}$, as in Theorem \ref{thm:ScottIsomorphism}.  Since $\sigma$ is in $\IL^C$, we have
\[\sigma^N = \sigma^{\overline{N}} = \begin{cases}0 &\text{ if $\overline{M} \cong \overline{N}$} \\ 1 &\text{ otherwise}\end{cases}.\]
\end{proof}

The following observation should be compared with Example \ref{Comparison:ex:EEExample} and Proposition \ref{prop:ExactNegation}, which showed that there are $\IL$ formulas (and even sentences) that are not $\IL^C$ or $\IL^C(\rho)$ formulas.

\begin{prop}\label{prop:SameEERelation}
For any separable complete metric structures $M$ and $N$ in the same countable signature, the following are equivalent:
\begin{enumerate}
\item{
$M \cong N$,
}
\item{
$M \equiv_{\IL} N$,
}
\item{
$M \equiv_{\IL^C(\rho)} N$,
}
\item{
$M \equiv_{\IL^C} N$.
}
\end{enumerate}
\end{prop}

\begin{proof}
It is clear that (1) implies (2).  By Theorem \ref{thm:SameLogic} each $\IL^C(\rho)$ formula can be implemented as an $\IL$ formula, so (2) implies (3).  Similarly, each $\IL^C$ formula is an $\IL^C(\rho)$ formula, so (3) implies (4).  Finally, if $M \equiv_{\IL^C} N$ then, in particular, $N$ satisfies $M$'s Scott sentence, and both are complete separable metric structures, so $M \cong N$ by Theorem \ref{thm:ScottIsomorphism}.
\end{proof}

The formula creation rules for $\IL$ imply that if $\phi(\vec{x})$ is an $\IL$-formula in a signature with a constant symbol $c$, then the expression obtained by replacing each instance of $c$ by a new variable $y$ is an $\IL$-formula $\psi(\vec{x}, y)$.  In particular, the usual identification of formulas with sentences in a language with new constant symbols can be used in $\IL$.  By contrast, when this procedure is performed on an $\IL^C$ or $\IL^C(\rho)$ formula, the result is not necessarily again an $\IL^C$ or $\IL^C(\rho)$ formula, because it may not have the appropriate continuity property.  Scott's isomorphism theorem provides a plentiful supply of examples.  

\begin{ex}\label{ex:Substitute}
Let $M$ be a complete separable connected metric structure such that $\Aut(M)$ does not act transitively on $M$.  Pick any $a \in M$, and let $\mc{O}_a$ be the $\op{Aut}(M)$-orbit of $a$.  The fact that $\Aut(M)$ does not act transitively implies $\mc{O}_a \neq M$.  Let $\theta_a(x)$ be the $\IL$ formula obtained by replacing $a$ by a new variable $x$ in the $\IL^C$ Scott sentence of $(M, a)$.  Then for any $b \in X$,
\begin{align*}
\theta_a^M(b) &= \begin{cases}0 & \text{ if $(M, a) \cong (M, b)$} \\ 1 & \text{ otherwise}\end{cases} \\
&= \begin{cases}0 & \text{ if $b \in \mc{O}_a$,} \\ 1 &\text{otherwise.}\end{cases}
\end{align*}
Since $M$ is connected and the image of $\theta_a^M$ is $\{0, 1\}$, the function $\theta_a^M$ is not continuous.  Therefore $\theta_a$ is not an $\IL^{C}$ or $\IL^C(\rho)$ formula.
\end{ex}

The fact that the formula $\theta_a$ in the above example \emph{is} an $\IL$ formula will be relevant in the proof of Theorem \ref{Infinitary:Scott:thmScottDefinability} below.

\section{Definability in $\IL$}\label{sec:Definability}
The original use of Scott's isomorphism theorem in \cite{Scott1965} was to prove a definability theorem.  We obtain an analogous definability theorem for the metric logic $\IL$.

\begin{thm}\label{Infinitary:Scott:thmScottDefinability}
Let $M$ be a separable complete metric structure in a countable signature.  For any continuous function $P : M^n \to [0, 1]$, the following are equivalent:
\begin{enumerate}
\item{
There is an $\IL$ formula $\phi(\vec{x})$ such that for all $\vec{a} \in M^n$, $\phi^M(\vec{a}) = P(\vec{a})$,
}
\item{
$P$ is fixed by all automorphisms of $M$ (in the sense that for all $\Phi \in \op{Aut}(M)$, $P = P \circ \Phi$).
}
\end{enumerate}
\end{thm}
\begin{proof}
The proof that (1) implies (2) is a routine induction on the complexity of formulas, so we only prove that (2) implies (1).

Fix a countable dense subset $D \subseteq M$.  For each $\vec{a} \in D$, let $\theta_{\vec{a}}(\vec{x})$ be the formula obtained by replacing each occurrence of $\vec{a}$ in the Scott sentence of $(M, \vec{a})$ by a tuple of new variables $\vec{x}$.  The Scott sentence is obtained from Theorem \ref{thm:ScottIsomorphism}.  Observe that this formula has the following property, for all $\vec{b} \in M^n$:
\[\theta_{\vec{a}}^M(\vec{b}) = \begin{cases}0 & \text{ if there is $\Phi \in \Aut(M)$ with $\Phi(\vec{b}) = \vec{a}$} \\ 1 & \text{otherwise}\end{cases}\]
For each $\epsilon > 0$, define:
\[\sigma_\epsilon(\vec{x}) = \inf_{\vec{y}} \max\set{d(\vec{x}, \vec{y}), \inf_{\substack{\vec{a} \in D^n \\ P(\vec{a}) < \epsilon}}\theta_{\vec{a}}(\vec{y})}.\]
Each $\sigma_\epsilon(\vec{x})$ is a formula of $\IL(S)$.
\begin{claim}\label{Infinitary:Scott:DefinabilityClaim}
Consider any $\epsilon \in \mathbb{Q} \cap (0, 1)$ and any $\vec{b} \in M^n$.
\begin{enumerate}
\item[(a)]{
If $M \models \sigma_\epsilon(\vec{b})$, then $P(\vec{b}) \leq \epsilon$.
}
\item[(b)]{
If $P(\vec{b}) < \epsilon$, then $M \models \sigma_\epsilon(\vec{b})$.
}
\end{enumerate}
\end{claim}
\begin{proof}
\begin{enumerate}
\item[(a)]{
Suppose that $M \models \sigma_\epsilon(\vec{b})$.  Fix $\epsilon' > 0$, and pick $0 < \delta < 1$ such that if $d(\vec{b}, \vec{y}) < \delta$, then $\abs{P(\vec{b}) - P(\vec{y})} < \epsilon'$.  This exists because we assumed that $P$ is continuous.  Now from the definition of $M \models \sigma_\epsilon(\vec{b})$ we can find $\vec{y} \in M^n$ such that 
\[\max\set{d(\vec{b}, \vec{y}), \inf_{\substack{\vec{a} \in D^n \\ P(\vec{a}) < \epsilon}}\theta_{\vec{a}}(\vec{y})} < \delta.\]
In particular, we have that $d(\vec{b}, \vec{y}) < \delta$, so $\abs{P(\vec{b}) - P(\vec{y})} < \epsilon'$.
 On the other hand, $\inf_{\substack{\vec{a} \in D^n \\ P(\vec{a}) < \epsilon}}\theta_{\vec{a}}(\vec{y}) < \delta$, and 
 $\theta_{\vec{a}}(\vec{y}) \in \{0, 1\}$ for all $\vec{a} \in D^n$, so in fact there is $\vec{a} \in D^n$ 
 with $P(\vec{a}) < \epsilon$ and $\theta_{\vec{a}}(\vec{y}) = 0$.  For such an $\vec{a}$ there is an automorphism of $M$ taking $\vec{y}$ to $\vec{a}$, and hence by (2) we have that $P(\vec{y}) < \epsilon$ as well.  Combining what we have,
\begin{align*}
P(\vec{b}) &= \abs{P(\vec{b})} \\
 &\leq \abs{P(\vec{b}) - P(\vec{y})} + \abs{P(\vec{y})} \\
 &< \epsilon' + \epsilon
\end{align*}
Taking $\epsilon' \to 0$ we conclude $P(\vec{b}) \leq \epsilon$.
}
\item[(b)]{
Suppose that $P(\vec{b}) < \epsilon$, and again fix $\epsilon' > 0$.  Using the continuity of $P$, find $\delta$ sufficiently small so that if $d(\vec{b}, \vec{y}) < \delta$ then $P(\vec{y}) < \epsilon$.  The set $D$ is dense in $M$, so we can find $\vec{y} \in D^n$ such that $d(\vec{b}, \vec{y}) < \min\{\delta, \epsilon'\}$.  Then $P(\vec{y}) < \epsilon$, so choosing $\vec{a} = \vec{y}$ we have
\[\inf_{\substack{\vec{a} \in D^n \\ P(\vec{a}) < \epsilon}}\theta_{\vec{a}}(\vec{y}) = 0.\]
Therefore
\[\max\set{d(\vec{b}, \vec{y}), \inf_{\substack{\vec{a} \in D^n \\ P(\vec{a}) < \epsilon}}\theta_{\vec{a}}(\vec{y})} = d(\vec{b}, \vec{y}) < \epsilon',\]
and so taking $\epsilon' \to 0$ shows that $M \models \sigma_{\epsilon}(\vec{b})$.
}
\end{enumerate}
\renewcommand{\qedsymbol}{$\dashv$ - Claim \ref{Infinitary:Scott:DefinabilityClaim}}
\end{proof}
Consider now any $\vec{a} \in M^n$.  By (a) of the claim $P(\vec{a})$ is a lower bound for $\set{\epsilon \in \mathbb{Q} \cap (0, 1) : M\models \sigma_\epsilon(\vec{a})}$.  If $\alpha$ is another lower bound, and $\alpha > P(\vec{a})$, then there is $\epsilon \in \mathbb{Q} \cap (0, 1)$ such that $P(\vec{a}) < \epsilon < \alpha$.  By (b) of the claim we have $M \models \sigma_\epsilon(\vec{a})$ for this $\epsilon$, contradicting the choice of $\alpha$.  Therefore
\[P(\vec{a}) = \inf\set{\epsilon \in \mathbb{Q} \cap (0, 1) : M \models \sigma_\epsilon(\vec{a})}.\] 

Now for each $\epsilon \in \mathbb{Q} \cap (0, 1)$, define a formula
\[\psi_\epsilon(\vec{x}) = \max\set{\epsilon, \sup_{m \in \mathbb{N}}\min\set{m\sigma_\epsilon(\vec{x}), 1}}.\]
Then for any $\vec{a} \in M^n$,
\[\psi_\epsilon^M(\vec{a}) = \begin{cases}\epsilon &\text{if $\sigma_\epsilon^M(\vec{a}) = 0$,} \\ 1 &\text{otherwise.}\end{cases}\]
Let $\phi(\vec{x}) = \inf_{\epsilon \in \mathbb{Q} \cap (0, 1)}\psi_\epsilon(\vec{x})$.  Then
\[\phi^M(\vec{a}) = \inf\set{\epsilon : \sigma_\epsilon^M(\vec{a}) = 0} = P(\vec{a}).\]
\end{proof}

We also have a version where parameters are allowed in the definitions:

\begin{cor}\label{Infinitary:Scott:corScottParameters}
Let $M$ be a separable complete metric structure in a countable signature, and fix a set $A \subseteq M$.  For any continuous function $P : M^n \to [0, 1]$, the following are equivalent:
\begin{enumerate}
\item{
There is an $\IL$ formula $\phi(\vec{x})$ with parameters from $A$ such that for all $\vec{a} \in M^n$, 
\[\phi^M(\vec{a}) = P(\vec{a}),\]
}
\item{
$P$ is fixed by all automorphisms of $M$ that fix $A$ pointwise,
}
\item{
$P$ is fixed by all automorphisms of $M$ that fix $\overline{A}$ pointwise.
}
\end{enumerate}
\end{cor}

\begin{proof}
Since $M$ is a separable metric space, there is a countable set $D \subseteq A$ such that $\overline{D} = \overline{A}$ in $M$.  An automorphism of $M$ fixes $A$ pointwise if and only if it fixes $D$ pointwise if and only if it fixes $\overline{D} = \overline{A}$ pointwise, which establishes the equivalence of (2) and (3).
For the equivalence of (1) and (2), apply Theorem \ref{Infinitary:Scott:thmScottDefinability} to the structure obtained from $M$ by adding a new constant symbol for each element of $D$.
\end{proof}

Theorem \ref{Infinitary:Scott:thmScottDefinability} does not hold, as stated, with $\IL$ replaced by $\IL^C$ or $\IL^C(\rho)$, because we assumed only continuity for the function $P$, while formulas in $\IL^C$ always define uniformly continuous functions.  Even if $P$ is assumed to be uniformly continuous, some intermediate steps in our proof use the formulas discussed in Example \ref{ex:Substitute}, as well as other possibly discontinuous formulas, and hence our argument does not directly apply to give a version of Scott's definability theorem in the other infinitary logics.

\begin{q}\label{q:DefinabilityInILC}
Let $M$ be a separable complete metric structure, and let $P : M^n \to [0, 1]$ be uniformly continuous and automorphism invariant.  Is $P$ definable in $M$ by an $\IL^C$-formula?  Is $P$ definable in $M$ by an $\IL^C(\rho)$-formula?
\end{q}

To conclude, we give one quite simple example of definability in $\IL$ where first-order definability fails.

\begin{ex}\label{ex:CStar}
Recall that a (unital) C*-algebra is a unital Banach algebra with an involution $*$ satisfying the C*-identity $\norm{xx^*} = \norm{x}^2$.  A formalization for treating C*-algebras as metric structures is presented in \cite{Farah2014b}, where it is also shown that in an appropriate language the class of C*-algebras is $\forall$-axiomatizable in continuous first-order logic.  The model theory of C*-algebras has since become an active area of investigation.

A \defined{trace} on a C*-algebra $A$ is a bounded linear functional $\tau : A \to \mathbb{C}$ such that $\tau(1) = 1$, and for all $a, b \in A$, $\tau(a^*a) \geq 0$ and $\tau(ab)=\tau(ba)$.   An appropriate way to consider traces as $[0, 1]$-valued predicates on the metric structure associated to a C*-algebra is given in \cite{Farah2014b}.  Traces appear as important tools throughout the C*-algebra literature.  In the first-order continuous model theory of C*-algebras, traces play a key role in showing that certain important C*-algebras can be constructed as Fra\"iss\'e limits \cite{Eagle2014b}, and traces are also related to the failure of quantifier elimination for most finite-dimensional C*-algebras \cite{Eagle2015}.  Several other uses of traces in the model theory of C*-algebras can be found in \cite{Farah}.  Of particular interest is the case where a C*-algebra has a unique trace; such algebras are called \emph{monotracial}. 

In general, traces on C*-algebras need not be automorphism invariant.  For an example, consider $C(2^{\omega})$, the C*-algebra of continuous complex-valued functions on the Cantor space.  Pick any $z \in 2^{\omega}$, and define $\tau : C(2^{\omega}) \to \mathbb{C}$ by $\tau(f) = f(z)$.  It is straightforward to verify that $\tau$ is a trace.  For any other $z' \in 2^{\omega}$ there is an autohomeomorphism $\phi$ of $2^{\omega}$ sending $z$ to $z'$.  The map $\Phi : f \mapsto f \circ \phi$ is then an automorphism of $C(2^{\omega})$, and we have $(\tau \circ \Phi)(f) = \tau(f \circ \phi) = (f \circ \phi)(z) = f(z')$, so $\tau \circ \Phi \neq \tau$.

On the other hand, it is easily seen that if $\tau$ is a trace on $A$ and $\Phi \in \op{Aut}(A)$, then $\tau \circ \Phi$ is again a trace on $A$.  Thus for monotracial C*-algebras the unique trace \emph{is} automorphism invariant.  The following is therefore a direct consequence of Theorem \ref{Infinitary:Scott:thmScottDefinability}:

\begin{cor}\label{cor:CStar}
If $A$ is a separable C*-algebra with a unique trace $\tau$, then $\tau$ is $\IL$-definable (without parameters) in $A$.
\end{cor}

It is natural to ask whether $\IL$-definability in Corollary \ref{cor:CStar} can be replaced by definability in a weaker logic.  Monotracial C*-algebras satisfying certain additional properties do have their traces definable in first-order continuous logic (see \cite{Farah}), but the additional assumptions on the C*-algebras are necessary.  In \cite{Farah} it is shown that the separable monotracial C*-algebra constructed by Robert in \cite[Theorem 1.4]{Robert2013} has the property that the trace is not definable in first-order continuous logic.

The situation for definability in $\IL^C$ is less clear.  Any trace on a C*-algebra is $1$-Lipschitz, and so in particular is uniformly continuous.  An interesting special case of Question \ref{q:DefinabilityInILC} is then whether or not the trace on a monotracial separable C*-algebra is always $\IL^C$-definable.
\end{ex}
\bibliographystyle{amsalpha}
\bibliography{ExpressivePower}
\end{document}